\documentclass[final,4p,4pt, times]{article}
\usepackage{authblk}
\usepackage{amsmath, amsthm, mathrsfs, dsfont}
\usepackage{amssymb,latexsym}
\usepackage[german,english]{babel}
\usepackage{url}
\usepackage{graphicx}
\usepackage{gastex}
\usepackage{longtable}
\usepackage{lscape}
\usepackage{tabularx}
\usepackage{multicol}
\usepackage{verbatim}
\usepackage{multirow}
\usepackage{graphicx}
\usepackage{blkarray}   
\usepackage{epsfig,graphics,graphicx}
\usepackage{geometry}
\usepackage{setspace}
\usepackage[hidelinks]{hyperref}
\usepackage{orcidlink}
\usepackage{tikz}
\usepackage{upgreek}
\usepackage{cite}

\usetikzlibrary{decorations.pathreplacing}
\usepackage{amsmath,blkarray,booktabs,bigstrut}
\usepackage{graphicx}
\newcommand\undermat[2]{\makebox[0pt][l]{$\smash{\underbrace{\phantom{%
					\begin{matrix}#2\end{matrix}}}_{\text{$#1$}}}$}#2}

\hbadness=\maxdimen
\newtheorem{eg}{{\bf Example}}[subsection]
\newtheorem{rem}{{\bf Remark}}

\newtheorem{thm}{{\bf Theorem}}[section]
\newtheorem{prop}{{\bf Proposition}}[section]
\newtheorem{cor}{{\bf Corollary}}[section]

\usepackage{chngcntr}
\counterwithout{equation}{section}
\usepackage{mathrsfs}

\begin{document}
\title{$RD_\alpha$-Spectra of Joined Union Graphs with Applications to Power Graphs of Finite Groups} 
\author{Aditya Singh$^1$\orcidlink{0009-0000-8323-4563} \ Yogendra Singh$^2$\orcidlink{0000-0001-6305-7168} \ 
Anand Kumar Tiwari$^1$\orcidlink{0000-0002-7038-1801}}

\date{
$^1$\small Department of Applied Sciences, Indian Institute of Information Technology, Allahabad 211015, India \\
$^2$\small Department of Mathematics and Statistics, Vignan’s Foundation for Science, Technology \& Research, Vadlamudi, 522213, India \\
\vspace{1em}
\textbf{Corresponding Author:} Anand Kumar Tiwari \\
\hspace{2.4cm} \textbf{Email:} \texttt{anand@iiita.ac.in}
}

\maketitle
	
\hrule
\begin{abstract}
The \emph{generalized reciprocal distance matrix} of a graph $\mathscr{G}$, denoted by $RD_\alpha(\mathscr{G})$, is defined as $RD_\alpha(\mathscr{G})=\alpha\,RT_r(\mathscr{G})+(1-\alpha)\,RD(\mathscr{G}), \, \alpha\in[0,1],$ where $RT_r(\mathscr{G})$ represents the diagonal matrix of reciprocal vertex transmissions, and $RD(\mathscr{G})$ is the Harary (reciprocal distance) matrix of $\mathscr{G}$. In this paper, we investigate the $RD_\alpha$-spectrum of graphs obtained through the joined union operation. We derive explicit formulas for the characteristic polynomial of $RD_\alpha(\mathscr{G})$ when $\mathscr{G}$ is formed as a joined union of regular graphs. These results provide closed-form expressions for the corresponding spectra of several important graph classes. Moreover, we show that the power graphs of the dihedral group $D_{2n}$ and the generalized quaternion group $Q_{4n}$ admit representations as joined union graphs. Using this structural characterization, we determine the $RD_\alpha$-spectra of power graphs arising from various classes of finite groups, including cyclic groups $\mathbb{Z}_n$, dihedral groups $D_{2n}$, generalized quaternion groups $Q_{4n}$, elementary abelian $p$-groups, and certain non-abelian groups of order $pq$.
\end{abstract}


\hrule

\section{Introduction} In this paper, we focus on finite, simple, and connected graph. Let $\mathscr{G}=(V(\mathscr{G}), E(\mathscr{G}))$ be a graph with a vertex set $V(\mathscr{G})$ and an edge set $E(\mathscr{G})$. Two vertices $u$ and $v$ are adjacent, denoted by $u\sim v$, if $\{u, v\} \in E(\mathscr{G})$. The neighborhood of a vertex \( v\), denoted by \( \mathscr{N}(v) \), is the set \( \{ u \in V(\mathscr{G}) : u \sim v \} \), and its degree is defined as \( \deg (v) = |\mathscr{N}(v)| \). The graph $\mathscr{G}$ is called $k$-regular if every vertex has degree $k$. For more graph-theoretic terms, one can see \cite{bh(2010)}.

Graphs can be represented by several matrices, which have important applications in areas such as theoretical physics, quantum mechanics, molecular chemistry, and communication networks \cite{cdt(1980), jmnt(2007)}. Among various graph matrices, distance-based matrices play a central role in chemical graph theory, where many molecular properties depend on distances between atoms. One important distance-based invariant is the Harary index \cite{pntm(1993)}, which is obtained from the reciprocal distance matrix and gives greater weight to pairs of vertices that are closer together. The \emph{reciprocal distance matrix}, introduced by Plav\v{s}i\'c \emph{et al.} \cite{pntm(1993)}, is defined for a simple connected graph $\mathscr{G}$ as \[ RD(\mathscr{G})=(a_{ij}) = \begin{cases} \displaystyle \frac{1}{d(u_i,u_j)}, & \text{if } u_i \neq u_j,\\[6pt] 0, & \text{otherwise},\end{cases} \] where $d(u_i,u_j)$ denotes the shortest path distance between $u_i$ and $u_j$ in $\mathscr{G}$. 
The \emph{reciprocal transmission} of a vertex $v \in V(\mathscr{G})$ 
is defined as
\[
RT_r(v)=\sum_{u \in V(\mathscr{G}) \setminus \{v\}} 
\frac{1}{d(v,u)},
\]
and the corresponding diagonal matrix of reciprocal transmissions is
\[
RT_r(\mathscr{G})=
\operatorname{diag}\big(RT_r(v_1),RT_r(v_2),\ldots,
RT_r(v_{|V(\mathscr{G})|})\big).
\]

Using these matrices, the \emph{reciprocal distance Laplacian matrix} and the \emph{reciprocal distance signless Laplacian matrix} were introduced by Bapat and Panda \cite{bp(2018)} and later by Alhevaz \emph{et al.} \cite{abr(2019)}, respectively. These matrices are defined as $RL(\mathscr{G}) = RT_r(\mathscr{G}) - RD(\mathscr{G}) \quad \text{and} \quad RQ(\mathscr{G}) = RT_r(\mathscr{G}) + RD(\mathscr{G}).$ To formulate a generalized version of these matrices, Tian \emph{et al.} \cite{tcc} introduced the \emph{generalized reciprocal distance matrix} \( RD_{\alpha}(\mathscr{G}) \), defined as
\[
RD_{\alpha}(\mathscr{G}) = \alpha RT_r(\mathscr{G}) + (1 - \alpha) RD(\mathscr{G}), \ \text{where } \alpha \in [0,1].
\]
This convex combination matrix provides a flexible framework that captures both local (via \( RT_r(\mathscr{G})\)) and global (via \( RD(\mathscr{G}) \)) distance-based features of a graph. Notably, for specific values of \( \alpha \), it recovers key matrices: \( RD_{\alpha}(\mathscr{G}) = RD(\mathscr{G}) \) when \( \alpha = 0 \), \( RD_{\alpha}(\mathscr{G}) = RT_r(\mathscr{G}) \) when \( \alpha = 1 \), and \( RD_{\alpha}(\mathscr{G}) = RQ(\mathscr{G})/2 \) when \( \alpha = 0.5 \). Since the generalized reciprocal distance matrix \( RD_{\alpha}(\mathscr{G}) \) of a graph \( \mathscr{G} \) is real and symmetric, all its eigenvalues are real and can be ordered as 
\[
\lambda_1(RD_{\alpha}(\mathscr{G})) \geq \lambda_2(RD_{\alpha}(\mathscr{G})) \geq \cdots \geq \lambda_n(RD_{\alpha}(\mathscr{G})).
\]

The largest eigenvalue, \( \lambda_1(RD_{\alpha}(\mathscr{G})) \), is known as the \emph{spectral radius} of \( RD_{\alpha}(\mathscr{G}) \) and is denoted by \( \rho(RD_{\alpha}(\mathscr{G})) \). We denote by $J_n$ and $I_n$ the all-ones and identity matrices of order $n$, respectively. Further background on these matrices can be found in \cite{abr(2019), bp(2018), tcc, mt(2021), zt(2008)}.

Graph operations play an important role in spectral graph theory, as they allow complex graphs to be constructed from simpler ones using operations such as joins, unions, edge addition or deletion, and graph complements. These operations often make it possible to express the spectrum of a complicated graph in terms of the spectra of its basic components. A detailed survey of these methods can be found in Barik et al.\cite{bkps(2018)}.

Motivated by these ideas, in this paper we derive the characteristic polynomial of the generalized reciprocal distance matrix $RD_\alpha(\mathscr{G})$ for the joined union of regular graphs. This result allows us to determine the $RD_\alpha$-spectra of several important families of graphs. In particular, we express the power graphs of the dihedral group $D_{2n}$ and the generalized quaternion group $Q_{4n}$ as joined union graphs. As applications, we compute the generalized reciprocal distance spectra of power graphs associated with cyclic groups $\mathbb{Z}_n$, dihedral groups $D_{2n}$, generalized quaternion groups $Q_{4n}$, elementary abelian $p$-groups, and certain non-abelian groups of order $pq$. The paper concludes with a summary of the main results and possible directions for future research.

\section{\texorpdfstring{$RD_\alpha$}{RD-alpha}-spectrum of joined union of graphs}

 The joined union graph $\mathscr{G}[\mathscr{G}_1, \mathscr{G}_2, \ldots, \mathscr{G}_n]$ of graphs $\mathscr{G}_1, \mathscr{G}_2, \ldots, \mathscr{G}_n$ with respect to another graph $\mathscr{G}$ is a graph obtained from the union of graphs $\mathscr{G}_1, \mathscr{G}_2, \ldots, \mathscr{G}_n$ by joining every vertex of $\mathscr{G}_i$ to every vertex of $\mathscr{G}_j$, whenever $v_i \sim v_j$ in $\mathscr{G}$. Here, each vertex $v_i \in V(\mathscr{G})$  corresponds to the entire graph $\mathscr{G}_i$.
In the case when  $n = 2$ and $\mathscr{G} = K_2$, the joined union graph $K_2[\mathscr{G}_1, \mathscr{G}_2]$ reduces to the join of the graphs $\mathscr{G}_1$ and $\mathscr{G}_2$, also denoted by $\mathscr{G}_1 \vee \mathscr{G}_2$.

Let $\mathscr{M}$ be an $n \times n$ matrix, and let its rows and columns be grouped according to a partition
$\mathscr{P} = \{\mathscr{P}_1, \mathscr{P}_2, \ldots, \mathscr{P}_m\},$
where each $\mathscr{P}_i \subseteq \mathscr{X}  = \{1, 2, \ldots, n\}$ and the union of all $\mathscr{P}_i$ gives the full index set $\mathscr{X}$. Then the matrix $\mathscr{M}$ can be written in block form as

$$\mathscr{M} =
\begin{pmatrix}
\mathscr{M}_{1,1} & \mathscr{M}_{1,2} & \cdots & \mathscr{M}_{1,m} \\
\mathscr{M}_{2,1} & \mathscr{M}_{2,2} & \cdots & \mathscr{M}_{2,m} \\
\vdots & \vdots & \ddots & \vdots \\
\mathscr{M}_{m,1} & \mathscr{M}_{m,2} & \cdots & \mathscr{M}_{m,m}
\end{pmatrix},$$
where each block $\mathscr{M}_{i,j}$ corresponds to the submatrix formed by selecting the rows indexed by $\mathscr{P}_i$ and the columns indexed by $\mathscr{P}_j$, for $1 \leq i, j \leq m$. The quotient matrix associated with this block structure is the $m \times m$ matrix $\mathcal{Q} = (\mathcal{Q}_{i,j})$, where each entry $\mathcal{Q}_{i,j}$  is defined as the average row sum of the block $\mathscr{M}_{i,j},$ (see \cite[Section 2.3]{bh(2010)}). The partition $\mathscr{P}$ is equitable if every block $\mathscr{M}_{i,j}$ has constant row sums (and hence constant column sums, since $\mathscr{M}$ is square). In this case, the quotient matrix $\mathcal{Q}$ is referred to as an equitable quotient matrix.

In general, the eigenvalues of the quotient matrix $\mathcal{Q}$ interlace those of $\mathscr{M}$. In this case, we have following useful result.
\begin{prop} {\bf (Brouwer and Haemers \cite{bh(2010)})}\label{p2.1}
If the partition $\mathscr{P}$ of $\mathscr{X}$ of a matrix $\mathscr{M}$ is equitable, then every eigenvalue of the quotient matrix $\mathcal{Q}$ is also an eigenvalue of $\mathscr{M}$.
\end{prop}

A connected graph $\mathscr{G}$ is called \emph{reciprocal transmission regular} if the reciprocal transmission of each vertex is the same, i.e., $RT_r(v)=k$, for each $v \in V(\mathscr{G})$. A connected graph $\mathscr{G}$ is said to be \emph{$m$-partitioned reciprocal transmission regular} if its vertex set admits a partition $V(\mathscr{G})=\displaystyle\bigcup_{i=1}^m V_i$ such that for any $i,j \in \{1,2,\dots,m\}$ and for every vertex $v\in V_i$, the quantity $q_{ij}=\displaystyle\sum_{u \in V_j}\frac{1}{d(v,u)}$ is constant, i.e., independent of the choice of $v$ in $V_i$. For an $m$-partitioned reciprocal transmission regular graph $\mathscr{G}$, the collection $\{V_1,V_2,\dots,V_m\}$ forms an equitable partition of the reciprocal distance matrix $RD(\mathscr{G})$. Hence, the corresponding quotient matrix is $\mathcal{Q}^{RD}=\big[q_{ij}\big]_{m\times m}.$ By Proposition \ref{p2.1}, every eigenvalue of $\mathcal{Q}^{RD}$ is also an eigenvalue of $RD(\mathscr{G})$. The quotient matrix of the \emph{reciprocal distance Laplacian matrix} $RL(\mathscr{G})$ with respect to this partition is given by $\mathcal{Q}^{RL}=\big[q_{ij}\big]_{m\times m},$ where
\[
q_{ij}=
\begin{cases}
\displaystyle RT_r(v)-\sum_{u \in V_i}\frac{1}{d(v,u)}, & \text{if } i=j,\\[8pt]
\displaystyle -\sum_{u \in V_j}\frac{1}{d(v,u)}, & \text{if } i\neq j,
\end{cases}
\]
for $v \in V_i$ and $v \neq u$. Similarly, the quotient matrix of the \emph{reciprocal distance signless Laplacian matrix}
$RQ(\mathscr{G})$ is $\mathcal{Q}^{RQ}=\big[q_{ij}\big]_{m\times m},$ with entries
\[
q_{ij}=
\begin{cases}
\displaystyle RT_r(v)+\sum_{u \in V_i}\frac{1}{d(v,u)}, & \text{if } i=j,\\[8pt]
\displaystyle \sum_{u \in V_j}\frac{1}{d(v,u)}, & \text{if } i\neq j.
\end{cases}
\]

Finally, for the \emph{generalized reciprocal distance matrix} $RD_{\alpha}(\mathscr{G})$, $\alpha\in[0,1]$, the quotient matrix $\mathcal{Q}^{RD_{\alpha}}=\big[q_{ij}\big]_{m\times m}$ is given by
\[
q_{ij}=
\begin{cases}
\displaystyle \alpha\,RT_r(v)+(1-\alpha)\sum_{u \in V_i}\frac{1}{d(v,u)}, & \text{if } i=j,\\[8pt]
\displaystyle (1-\alpha)\sum_{u \in V_j}\frac{1}{d(v,u)}, & \text{if } i\neq j.
\end{cases}
\]

We now present a theorem that describes the generalized reciprocal distance spectrum of the joined union graph.

\begin{thm} \label{t2.1}
Let $\mathscr{G}$ be a graph of order $n$ with the vertex set $V(\mathscr{G}) = \{v_1, v_2, \ldots, v_n\}$. For each $i = 1, 2, \ldots, n$, let $\mathscr{G}_i$ be an $r_i$-regular graph of order $n_i$, with the vertex set $V(\mathscr{G}_i) = \{v_{i1}, v_{i2}, \ldots, v_{in_i}\}$. Consider the joined union graph $\mathscr{G}[\mathscr{G}_1, \mathscr{G}_2, \ldots, \mathscr{G}_n]$, and let $\mathcal{Q}^{RD_\alpha}$ denote the quotient matrix of its generalized reciprocal distance matrix $RD_\alpha$ with respect to the equitable partition $\{V(\mathscr{G}_1), V(\mathscr{G}_2), \ldots, \linebreak V(\mathscr{G}_n)\}$. Then, the characteristic polynomial of $RD_\alpha$ is given by
\[
\operatorname{Char}(\mathscr{G}[\mathscr{G}_1, \mathscr{G}_2, \ldots, \mathscr{G}_n], x) = \operatorname{Char}(\mathcal{Q}^{RD_\alpha}, x) \prod_{i=1}^n \frac{\operatorname{Char}(\mathscr{G}_i, x)}{x - \left\{ \frac{1}{2}(n_i + r_i - 1) + m_i \right\} \alpha + (1 - \alpha)n_i'},
\]
where $n_i' = \sum_{j=2}^{n_i} \frac{1}{d(v_{i1}, v_{ij})}$, 
$m_i = \sum_{t=1}^{n} \frac{n_t}{d_{\mathscr{G}}(v_i, v_t)}$, and the quotient matrix $\mathcal{Q}^{RD_\alpha} = (q_{ij})_{n \times n}$ is defined by
\[
q_{ij} =
\begin{cases}
\displaystyle \alpha \left(\frac{n_i - r_i - 1}{2} + m_i\right) + (1 - \alpha) \sum_{j=2}^{n_i} \frac{1}{d(v_{i1}, v_{ij})}, & \text{if } i = j, \\
\displaystyle \frac{(1 - \alpha)n_j}{d_{\mathscr{G}}(v_i, v_j)}, & \text{if } i \neq j.
\end{cases}
\]

\end{thm}

\noindent{\textbf{Proof.}} Let $\mathscr{H}=\mathscr{G}[\mathscr{G}_1,\mathscr{G}_2,\ldots,\mathscr{G}_n]$ be the joined union of regular graphs $\mathscr{G}_1,\mathscr{G}_2,\ldots,\mathscr{G}_n$, where each $\mathscr{G}_i$ corresponds to a vertex $v_i$ of a connected graph $\mathscr{G}$ of order $n$. Then, for a fixed vertex $v_{ij} \in V(\mathscr{H})$, for $1 \leq i \leq n$ and $1 \leq j \leq n_i$, the reciprocal transmission degree is  
\begin{equation*}
  \begin{split}
RT_r(v_{ij}) &=\frac{n_1}{d_{\mathscr{G}}(v_i,v_1)}+\frac{n_2}{d_{\mathscr{G}}(v_i,v_2)}+\cdots+\frac{n_n}{d_{\mathscr{G}}(v_i,v_n)}+deg(v_{ij})+\frac{1}{2}(n_i-1-deg( v_{ij})), \\
& = \frac{1}{2}(n_i+r_i-1)+m_i , \text{ where } m_i=\sum_{t=1}^{n} \frac{n_t}{d_\mathscr{G}(v_{i},v_{t})} \text{ for } 1 \leq i \leq n \ \text{and} \ i \neq t.
\end{split}
\end{equation*}

Under an appropriate labeling of the vertices of the joined union graph $\mathscr{H}$, the generalized reciprocal distance matrix of $\mathscr{H}$, denoted by $RD_\alpha(\mathscr{H})$, is given as  

$$RD_\alpha(\mathscr{H})=\begin{pmatrix}
  \mathscr{H}_1 & \frac{(1-\alpha)J_{n_1 \times n_2}}{d_\mathscr{G}(v_1,v_2)} & \cdots &   \frac{(1-\alpha)J_{n_1 \times n_n}}{d_\mathscr{G}(v_1,v_n)}\\
   \frac{(1-\alpha)J_{n_2 \times n_1}}{d_\mathscr{G}(v_2,v_1)} & \mathscr{H}_2 & \cdots &  \frac{(1-\alpha)J_{n_2 \times n_n}}{d_\mathscr{G}(v_2,v_n)}\\
   \vdots & \vdots & \ddots & \vdots \\
    \frac{(1-\alpha)J_{n_n \times n_1}}{d_\mathscr{G}(v_n,v_1)} &  \frac{(1-\alpha)J_{n_n \times n_2}}{d_\mathscr{G}(v_n,v_2)} & \cdots & \mathscr{H}_n
\end{pmatrix},$$ where \[  \mathscr{H}_i = 
\left( h_{jk} \right) = 
\begin{cases}
\alpha \left( \dfrac{n_i + r_i - 1}{2} + m_i \right), & \text{if } j = k, \\[10pt]
\dfrac{1 - \alpha}{d(v_{ij}, v_{ik})}, & \text{if } j \ne k,
\end{cases}
\quad \text{for } 1 \le j, k \le n_i.
\] 

Since $\mathscr{G}_i$ is an $r_i$-regular graph, the matrix $\mathscr{H}_i$ has a constant row sum. Consequently, the all-ones vector $e_{n_i}=(1,1,\ldots,1)^T$ of length $n_i$ is an eigenvector of $\mathscr{H}_i$ corresponding to this constant eigenvalue. The remaining $n_i-1$ eigenvectors are orthogonal to $e_{n_i}$. Let $\mathscr{X}=\{x_{i1},x_{i2},\ldots,x_{in_i}\}^t$ be an eigenvector of $\mathscr{H}_i$ such that $e^t_{n_i}\mathscr{X}=0$. This vector $\mathscr{X}$ can be viewed as a function on the vertex set $V(\mathscr{G}_i)$, associating each vertex $v_{ij}$ with the value $x_{ij}$, that is, $\mathscr{X}(v_{ij})=x_{ij}$ for $1 \leq i \leq n$ and $1 \leq j \leq n_i$. Now define the vector $\mathscr{Y}=\{y_1,y_2,\ldots,y_t\}^t$, where $t=\displaystyle\sum_{i=1}^{n}n_i$, as an eigenvector of the matrix $\mathscr{H}$ associated with the joined union graph. This vector is given by
$$y_j = \begin{cases}
    x_{ij} , & \text{if $v_{ij} \in V(\mathscr{G}_i)$},\\
    0, & \text{otherwise.}
    \end{cases}$$
This construction contributes $n_i-1$ eigenvectors from each $\mathscr{G}_i$, summing to a total of $\displaystyle\sum_{i=1}^{n}(n_i-1)=\left(\displaystyle\sum_{i=1}^{n}n_i\right)-n$ such eigenvectors. Additionally, the quotient matrix $\mathcal{Q}^{RD_\alpha}$ yields $n$ more eigenvectors. Together, these account for all $\displaystyle\sum_{i=1}^{n}n_i$ eigenvectors of the joined union graph $\mathscr{G}[\mathscr{G}_1,\mathscr{G}_2,\ldots,\mathscr{G}_n]$, with each subgraph $\mathscr{G}_i$ contributing one eigenvalue through the quotient structure. Now, each subgraph $\mathscr{G}_i$ contributes an eigenvalue 
$$\alpha\left\{\frac{n_i+r_i-1}{2}+m_i\right\}+(1-\alpha)\sum_{j_2=1}^{n_i}\frac{1}{d(v_{i1},v_{ij_2})}, \, \, \text{where} \ 1 \leq i \leq n,$$ which corresponds to the eigenvector $e_{n_i}$. To eliminate this eigenvalue from the characteristic polynomial $Char(\mathscr{G}_i,x)$, we divide by $$x-\alpha\left\{\frac{n_i+r_i-1}{2}+m_i\right\}-(1-\alpha)\sum_{j_2=1}^{n_i}\frac{1}{d(v_{i1},v_{ij_2})}.$$
Multiplying the resulting expression by the characteristic polynomial $Char(\mathcal{Q}^{RD_\alpha},x)$ of the quotient matrix gives the desired result.  \hfill$\blacksquare$

By putting $\alpha=0$ in Theorem \ref{t2.1}, we get the following result.

\begin{cor}\label{c2.1}
  Let $\mathscr{G}$ be a graph of order $n$ with  the vertex set $V(\mathscr{G})=\{v_1,v_2,\ldots,v_n\}$ and $\mathscr{G}_i$ be an $r_i$-regular graphs of order $n_i$. Let $\{V(\mathscr{G}_1),V(\mathscr{G}_2),\ldots,V(\mathscr{G}_n)\}$ is an equitable partition of $\mathscr{G}[\mathscr{G}_1,\mathscr{G}_2,\ldots,\mathscr{G}_n]$ and $\mathcal{Q}^{RD_0}$ denote the quotient matrix associated with this partition. Then the characteristic polynomial of the reciprocal distance matrix (Harary matrix) of $\mathscr{G}[\mathscr{G}_1,\mathscr{G}_2,\ldots,\mathscr{G}_n]$ is   
 $$Char(\mathscr{G}[\mathscr{G}_1,\mathscr{G}_2,\ldots,\mathscr{G}_n],x)=Char(\mathcal{Q}^{RD_0},x)\prod_{i=1}^{n}\frac{Char(\mathscr{G}_i,x)}{x-(\frac{n_i+r_i-1}{2})},$$  where the quotient matrix $\mathcal{Q}^{RD_0} = (q_{ij})_{n \times n}$ is defined by
\[
q_{ij} =
\begin{cases}
\displaystyle \sum_{j=2}^{n_i} \frac{1}{d(v_{i1}, v_{ij})}, & \text{if } i = j, \\
\displaystyle \frac{n_j}{d_{\mathscr{G}}(v_i, v_j)}, & \text{if } i \neq j.
\end{cases}
\]
\end{cor}

In the special case where each $\mathscr{G}_i$ is the complete graph $K_{q_{i}}$, we obtain the following result.

\begin{cor}\label{2.2}
Let $\mathscr{G}$ be a graph of order $n \geq 3$, and let $\mathscr{G}_i=K_{q_i}$ be a complete graph of order $q_i$. Then the characteristic polynomial of the generalized reciprocal distance matrix of the graph $\mathscr{G}[K_{q_1},K_{q_2},\ldots,K_{q_n}]$ is given by
$$Char(\mathscr{G}[K_{q_1},K_{q_2},\ldots,K_{q_n}],x)=Char(\mathcal{Q}^{RD_\alpha},x)\prod_{i=1}^{n}\frac{Char(K_{q_i},x)}{x-m'_i\alpha-p_i+1},$$ where $m'_i=\displaystyle\sum_{t=1}^{n} \frac{q_t}{d_\mathscr{G}(v_{i},v_{t})}$, $i \neq t$ and the quotient matrix $\mathcal{Q}^{RD_{\alpha}} = (q_{ij})_{n \times n}$ is defined by
\[
q_{ij} =
\begin{cases}
m'_i\alpha+p_i-1,  & \text{if } i = j, \\
\displaystyle \frac{(1-\alpha)p_j}{d_{\mathscr{G}}(v_i, v_j)}, & \text{if } i \neq j.
\end{cases}
\]
\end{cor}

Let \( K_{n_1, n_2, \ldots, n_q} \) denote the complete \( q \)-partite graph on \( \mathscr{N} = \sum_{i=1}^q n_i \) vertices. By the construction of the joined union graph, we observe that  
\[
K_{n_1, n_2, \ldots, n_q} = K_q[\overline{K}_{n_1}, \overline{K}_{n_2}, \ldots, \overline{K}_{n_q}].
\]
Applying Theorem \ref{t2.1} with \( \mathscr{G} = K_q \) and \( \mathscr{G}_i = \overline{K}_{n_i} \), and the distance between any two distinct vertices of \( \overline{K}_{n_i} \) in $K_{n_1, n_2, \ldots, n_q}$ is \( 2 \), i.e., \( d(v_{ij_1}, v_{ij_2}) = 2 \) for \( j_1 \neq j_2 \), and \( d_{K_q}(v_i, v_j) = 1 \) for all $1 \leq i,j \leq n$, along with \( m_i = \mathscr{N} - n_i \), we obtain the following result.

\begin{cor}\label{c2.3}
The characteristic polynomial of the generalized reciprocal distance matrix of the complete $q$-partite graph $K_{n_1,n_2,\ldots,n_q}$ with $\mathscr{N}=\sum_{i=1}^{q} n_i$ is given by $$Char(K_{n_1,n_2,\ldots,n_q},x)=Char(\mathcal{Q}^{RD_\alpha},x)\prod_{i=1}^{q}\frac{Char(\overline{K_{n_i}},x)}{x-\alpha(\mathscr{N}-n_i)-\frac{1}{2}(n_i-1)},$$ where the quotient matrix $\mathcal{Q}^{RD_{\alpha}} = (q_{ij})_{n \times n}$ is defined by
\[
q_{ij} =
\begin{cases}
\alpha(\mathscr{N}-n_i)-\frac{1}{2}(n_i-1), & \text{if } i = j, \\
 (1-\alpha)n_j, & \text{if } i \neq j.
\end{cases}
\]
\end{cor}

\begin{eg}
Let $\mathscr{G}=K_{n_1,n_2,\ldots,n_q}$ with $n_1=n_2=\cdots=n_q=n$. Then the generalized reciprocal matrix of $\mathscr{G}$ is given as $$RD_\alpha(\mathscr{G})=\begin{pmatrix}
     \mathscr{H}_1 & (1-\alpha)J_n & \cdots & (1-\alpha)J_n\\
     (1-\alpha)J_n & \mathscr{H}_2 & \cdots &  (1-\alpha)J_n\\
     \vdots & \vdots & \ddots & \vdots \\
      (1-\alpha)J_n &  (1-\alpha)J_n & \cdots & \mathscr{H}_q
 \end{pmatrix},$$ where $\mathscr{H}_i=\alpha(\mathscr{N}-\frac{n+1}{2})I_{n_i}+\frac{1-\alpha}{2}(J_{n_i}-I_{n_i})$, for $1\leq i \leq q$. Now, by using Corollary \ref{c2.3}, the characteristic polynomial of $RD_\alpha(\mathscr{G})$ is 

\[
\operatorname{Char}(RD_\alpha(\mathscr{G}), x)
= \operatorname{Char}\big(\mathcal{Q}^{RD_\alpha(\mathscr{G})}, x\big)
\left(x - n\alpha\left(q - \tfrac{1}{2}\right) + \tfrac{1}{2}\right)^{q(n-1)}, \]
 
where $\mathcal{Q}^{RD_\alpha(\mathscr{G})}$ is the corresponding quotient matrix given by
$$\mathcal{Q}^{RD_\alpha(\mathscr{G})}=\begin{pmatrix}
     n(q-1)\alpha+\frac{1}{2}(n-1) & (1-\alpha)n & \cdots & (1-\alpha)n\\
     (1-\alpha)n & n(q-1)\alpha+\frac{1}{2}(n-1) & \cdots &  (1-\alpha)n\\
     \vdots & \vdots & \ddots & \vdots \\
      (1-\alpha)n &  (1-\alpha)n & \cdots & n(q-1)\alpha+\frac{1}{2}(n-1) 
 \end{pmatrix}.$$
The characteristic polynomial of $\mathcal{Q}^{RD_\alpha(\mathscr{G})}$ is 
\[
\operatorname{Char}\big(\mathcal{Q}^{RD_\alpha(\mathscr{G})}, x\big)
= \left(x - nq +\tfrac{1}{2}(n +1)\right)
\left(x - \alpha qn + \tfrac{1}{2}(n+1)\right)^{q - 1}
.\]
Now, using the value of $Char(\mathcal{Q}^{RD_\alpha(\mathscr{G})},x)$ in $Char(\mathscr{G},x)$, we get 
\[
\operatorname{Char}\big(RD_\alpha(\mathscr{G}), x\big)
= \left(x - nq + \tfrac{1}{2}(n +1)\right)
\left(x - \alpha qn+ \tfrac{1}{2}(n + 1)\right)^{q - 1}
\left(x - n\alpha\left(q - \tfrac{1}{2}\right) + \tfrac{1}{2}\right)^{q(n - 1)}
.\]
\end{eg}

Next, we focus on computing the characteristic polynomial of the generalized reciprocal distance matrix for a graph formed by the join of two regular graphs, one of which is itself the disjoint union of two regular graphs on distinct vertex sets.

\begin{thm}\label{t2.3}
Let $\mathscr{G}_i$ be an $r_i$-regular graphs of order $n_i$ with the vertex set $V(\mathscr{G}_i) = \{v_{i1}, v_{i2}, \ldots, v_{in_i}\}$ for $i = 1, 2, 3$. Consider the graph $\mathscr{G} = \mathscr{G}_1 \vee (\mathscr{G}_2 \cup \mathscr{G}_3)$, and let $\mathscr{N} = n_1 + n_2 + n_3$. Then the characteristic polynomial of the generalized reciprocal distance matrix $Char(\mathscr{G}, x)$ is given by
$$Char(\mathcal{Q}^{RD_\alpha(\mathscr{G})
    },x)\frac{Char(\mathscr{G}_1,x)}{x-\frac{(n_1-r_1+1)\alpha}{2}-\sum_{j=2}^{n_1}\frac{(1-\alpha)}{d(v_{11},v_{ij})}}\prod_{i=2}^{3}\frac{Char(\mathscr{G}_i,x)}{x-\frac{(2N-n_2-n_3+r_i-1)\alpha}{2}-\sum_{j=2}^{n_1}\frac{(1-\alpha)}{d(v_{11},v_{ij})}}.$$
Here, the quotient matrix $\mathcal{Q}^{RD_\alpha}(\mathscr{G})$ with respect to the equitable partition $\{\mathscr{G}_1, \mathscr{G}_2, \mathscr{G}_3\}$ is given by  
\[
\mathcal{Q}^{RD_\alpha}(\mathscr{G}) = \begin{pmatrix}
f_1 & (1 - \alpha)n_2 & (1 - \alpha)n_3 \\
(1 - \alpha)n_1 & f_2 & \frac{(1-\alpha)}{2}n_3 \\
(1 - \alpha)n_1 & \frac{(1-\alpha)}{2}n_2 & f_3
\end{pmatrix},
\]
where
$f_1 = \left( \mathscr{N} - \frac{n_1 - r_1 + 1}{2} \right)\alpha +\sum_{j=2}^{n_1} \frac{(1-\alpha)}{d(v_{11}, v_{1j})}$,
$f_i = \left( \mathscr{N} - \frac{n_2 + n_3 - r_i + 1}{2} \right)\alpha +\sum_{j=2}^{n_i} \frac{(1-\alpha)}{d(v_{i1}, v_{ij})}$, for $i = 2, 3.$

\end{thm}

\noindent{\textbf{Proof.}}
Consider the graph $\mathscr{G}=\mathscr{G}_1 \vee (\mathscr{G}_2 \ \cup \ \mathscr{G}_3)=K_{1,2}[\mathscr{G}_1,\mathscr{G}_2,\mathscr{G}_3]$, where each $\mathscr{G}_i$ is an $r_i$-regular graph of order $n_i$ for $i=1,2,3$. Clearly, $\mathscr{G}$ is a graph of diameter $2$ and its vertex set is given by $V(\mathscr{G})=\mathscr{G}_1 \cup \mathscr{G}_2 \cup \mathscr{G}_3$. If $v_{it}\in \mathscr{V(\mathscr{G}_i)}$ for $1 \leq t \leq n_i$, then the reciprocal transmission degree of each vertex of $\mathscr{G}_i$, for $i=1,2,3$, is 
$RT_r(v_{1t}) =\mathscr{N}-\frac{1}{2}\left\{n_1-r_1+1\right\}$, $RT_r(v_{2t})=\mathscr{N}-\frac{1}{2}\left\{n_2+n_3-r_2+1\right\}$, and $RT_r(v_{3t})=\mathscr{N}-\frac{1}{2}\left\{n_2+n_3-r_3+1\right\}$, respectively. Note that the reciprocal transmission degree of each vertex of $\mathscr{G}_i$ is same for each $i$. Thus, the generalized reciprocal distance matrix of $\mathscr{G}$ is

\[
RD_\alpha
(\mathscr{G})=\begin{pmatrix}
    \mathscr{C}_1 & (1-\alpha)J_{n_1 \times n_2} & (1-\alpha)J_{n_1 \times n_3}\\
    (1-\alpha)J_{n_2 \times n_1} & \mathscr{C}_2 & \frac{(1-\alpha)}{2}J_{n_2 \times n_3}\\
    (1-\alpha)J_{n_3 \times n_1} & \frac{(1-\alpha)}{2}J_{n_3 \times n_2} & \mathscr{C}_3
\end{pmatrix},
\] where

\[
\mathscr{C}_1 = (c_{jk}) =
\begin{cases}
\left( \mathscr{N} - \dfrac{n_1 - r_1 + 1}{2} \right) \alpha, & \text{if } j = k, \\[10pt]
\dfrac{1 - \alpha}{d(v_{1j}, v_{1k})}, & \text{if } j \ne k,
\end{cases}
\quad \text{for } 1 \le j, k \le n_1,\] 
and for $i = 2, 3,$

\[ \mathscr{C}_i = (d_{jk})=
\begin{cases}
\left( \mathscr{N} - \dfrac{n_2 + n_3 - r_i + 1}{2} \right) \alpha, & \text{if } j = k, \\[10pt]
\dfrac{1 - \alpha}{d(v_{ij}, v_{ik})}, & \text{if } j \ne k,
\end{cases}
\quad \text{for } 1 \le j, k \le n_2.
\] Now, proceeding similarly as in Theorem \ref{t2.1}, we arrive at the result. \hfill$\blacksquare$
 
\begin{eg}
Let $\mathscr{G}_i = K_{n_i}$ for $i = 1, 2, 3$, and consider the graph $\mathscr{G} = K_{n_1} \vee (K_{n_2} \cup K_{n_3})$, with total order $\mathscr{N} = n_1 + n_2 + n_3$. In this construction, each $\mathscr{G}_i$ is a complete graph, so for any two vertices $v_1, v_2 \in V(\mathscr{G}_i)$, the reciprocal distance is $1$. Similarly, for $v_1 \in V(\mathscr{G}_1)$ and $v_2 \in V(\mathscr{G}_2) \cup V(\mathscr{G}_3)$, the reciprocal distance is also $1$, since the join ensures direct adjacency. However, for $v_1 \in V(\mathscr{G}_2)$ and $v_2 \in V(\mathscr{G}_3)$, the reciprocal distance is $\frac{1}{2}$, as these vertices are connected via a common neighbor in $\mathscr{G}_1$. By setting $r_1 = n_1 - 1$, $r_2 = n_2 - 1$, and $r_3 = n_3 - 1$, the generalized reciprocal distance matrix of $\mathscr{G}$ is a block matrix with diagonal blocks given by
$\mathscr{C}_1 = (\mathscr{N} - 1)\alpha I_{n_1} + (1 - \alpha)(J_{n_1} - I_{n_1})$, $\mathscr{C}_2 = \left(\mathscr{N} - \frac{n_3}{2} - 1\right)\alpha I_{n_2} + (1 - \alpha)(J_{n_2} - I_{n_2})$, and $\mathscr{C}_3 = \left(\mathscr{N} - \frac{n_2}{2} - 1\right)\alpha I_{n_3} + (1 - \alpha)(J_{n_3} - I_{n_3})$ and the rest of the entries of the matrix are same as the above matrix given in the proof of the Theorem \ref{t2.3}. Hence, the characteristic polynomial of the generalized reciprocal distance matrix of $\mathscr{G}$ is given by
\[
\mathrm{Char}(\mathcal{Q}^{RD_\alpha}(\mathscr{G}), x) 
\left(x - \mathscr{N}\alpha + 1\right)^{n_1 - 1}
\left(x - \left(\mathscr{N} - \frac{n_3}{2}\right)\alpha + 1\right)^{n_2 - 1}
\left(x - \left(\mathscr{N} - \frac{n_2}{2}\right)\alpha + 1\right)^{n_3 - 1},
\]
where the quotient matrix $\mathcal{Q}^{RD_\alpha}(\mathscr{G})$ is given by
\[
\mathcal{Q}^{RD_\alpha}(\mathscr{G}) =
\begin{pmatrix}
(\mathscr{N} - n_1)\alpha + n_1 - 1 & (1 - \alpha)n_2 & (1 - \alpha)n_3 \\
(1 - \alpha)n_1 & \left(n_1 + \frac{n_3}{2}\right)\alpha + n_2 - 1 & \frac{(1 - \alpha)}{2}n_3 \\
(1 - \alpha)n_1 & \frac{(1 - \alpha)}{2}n_2 & \left(n_1 + \frac{n_2}{2}\right)\alpha + n_3 - 1
\end{pmatrix}.
\]
\end{eg}

\section{\texorpdfstring{$RD_\alpha$}{RD-alpha}-spectrum of the power graphs of finite groups}

Given a finite group $\mathcal{G}$ of order $n$, Chakrabarty et al.\ \cite{cgs(2009)} introduced the undirected power graph $P(\mathcal{G})$, whose vertex set is $\mathcal{G}$, with two distinct elements $x, y \in \mathcal{G}$ being adjacent if one is a positive power of the other. For recent development of power graphs one can refer \cite{bp(2023), mga(2017), cs(2013), ba(2023), b(2023), bgp(2023),ss(2025),stpa(2025)}.

A divisor $d$ of $n$ is said to be a proper divisor if it satisfies $1 < d < n$.  Let the set of all such distinct proper divisors be denoted by $\{d_1, d_2, \ldots, d_t\}$. Define a simple graph  $\mathscr{T}$ with the vertex set $V(\mathscr{T}) = \{d_1, d_2, \ldots, d_t\}$, where two distinct vertices $d_i$ and $d_j$ are adjacent if and only if $d_i$ divides $d_j$, for $1 \leq i < j \leq t$. As established in \cite{bp(2023)}, the graph $\mathscr{T}$ is connected if and only if $n$ is neither a prime number nor a product of two distinct primes. By elementary number theory, for $n= p_1^{n_1} p_2^{n_2} \cdots p_r^{n_r}$, the total number of divisors of a positive integer $n$ is $\prod_{i=1}^{r} (n_i + 1)$,  where $p_1, p_2, \ldots, p_r$ are distinct primes and $n_1, n_2, \ldots, n_r$ are positive integers.
Thus, the order of the graph $\mathscr{T}$ is $|V(\mathscr{T})| =\prod_{i=1}^{r}(n_i + 1) - 2$. Observed that the power graph of a group $\mathcal{G}$ can be described Fusing the notion of a joined union graph with the help of the graph $\mathscr{T}$. For a positive integer $n>1$,  Euler's totient function $\phi(n)$ counts the integers in $\{1, 2, \ldots, n-1\}$ coprime to $n$.

\subsection{\texorpdfstring{$RD_\alpha$}{RD-alpha}-spectrum of the power graph of \texorpdfstring{$\mathbb{Z}_{n}$}{(Zn)}}
The following proposition describes the power graph of a finite cyclic group $\mathbb{Z}_n$ in terms of a joined union graph.

\begin{prop}\label{p3.1}{\bf (Mehranian et al. \cite{mga(2017)})}
If $\mathbb{Z}_n$ is a finite cyclic group of order $n \geq 3$, then the power graph $P(\mathbb{Z}_n)$ has the following form
$$P(\mathbb{Z}_n) = K_{\phi(n)+1} \vee \mathscr{T}[K_{\phi(d_1)},K_{\phi(d_2)},\ldots,K_{\phi(d_t)}].$$ 
\end{prop}

Using the above result, we determine the spectrum of the generalized reciprocal distance matrix associated to the power graph of $\mathbb{Z}_n$.

\begin{thm}\label{t3.1}
The generalized reciprocal distance spectra of the power graph of the finite cyclic group \( \mathbb{Z}_{n} \) consist of the eigenvalues $n\alpha - 1$ and $\left(\phi(n) + \phi(d_i) + \sum_{\substack{k = 2 \\ k \ne i}}^t \frac{\phi(d_k)}{d(v_i, v_k)}\right)\alpha -1 1$ with multiplicities $\phi(n)$ and $\phi(d_i)-1$, for $1 \leq i \leq t$, respectively. The remaining $t+1$ eigenvalues are given by the eigenvalues of the quotient matrix \( \mathcal{Q}^{RD_\alpha} \), which is defined as  
\[
\mathcal{Q}^{RD_\alpha} =
\begin{pmatrix}
\alpha(n - 1 - \phi(n)) + \phi(n) & (1 - \alpha)\phi(d_1) & \cdots & (1 - \alpha)\phi(d_t) \\
(1 - \alpha)(\phi(n) + 1) & h_1 & \cdots & \frac{(1 - \alpha)\phi(d_t)}{d(v_1, v_t)} \\
\vdots & \vdots & \ddots & \vdots \\
(1 - \alpha)(\phi(n) + 1) & \frac{(1 - \alpha)\phi(d_1)}{d(v_t, v_1)} & \cdots & h_t \\
\end{pmatrix},
\]
with
$h_i = \alpha\left(\phi(n) + \displaystyle\sum_{\substack{k = 1 \\ k \ne i}}^t \frac{\phi(d_k)}{d(v_i, v_k)} + 1\right) + \phi(d_i) - 1.$
\end{thm}

\noindent{\textbf{Proof.}} By Proposition \ref{p3.1}, the power graph of finite cyclic group $\mathbb{Z}_n$ can be written as $$P(\mathbb{Z}_n)=\mathscr{S}[K_{\phi(n)+1},K_{\phi(d_1)},\ldots,K_{\phi(d_t)}],$$ where $\mathscr{S}=K_{\phi(n)+1}\vee \mathscr{T}$ with the vertex set $\{v',v_1,v_2,\ldots,v_t\}$ and $t$ is the number of proper divisors of $n$. Using suitable labeling of the vertices of $P(\mathbb{Z}_n)$, the generalized reciprocal distance matrix is given by 

\[
RD_\alpha(P(\mathbb{Z}_{n}))=
\bordermatrix{
      & \phi(n)+1 & \phi(d_1) & \phi(d_2) & \cdots & \phi(d_t) \cr
\phi(n)+1    & \mathscr{H}' & (1-\alpha)J & (1-\alpha)J & \cdots & (1-\alpha)J \cr
\phi(d_1) & (1-\alpha)J & \mathscr{H}_1 & \frac{(1-\alpha)J}{d(v_1,v_2)} & \cdots & \frac{(1-\alpha)J}{d(v_1,v_t)} \cr
\phi(d_2) & (1-\alpha)J & \frac{(1-\alpha)J}{d(v_2,v_1)} & \mathscr{H}_2 & \cdots & \frac{(1-\alpha)J}{d(v_2,v_t)} \cr
\vdots & \vdots & \vdots & \vdots & \ddots & \vdots  \cr
\phi(d_t) & (1-\alpha)J & \frac{(1-\alpha)J}{d(v_t,v_1)} & \frac{(1-\alpha)J}{d(v_t,v_2)} & \cdots & \mathscr{H}_t
},
\]
 where $\mathscr{H}'=\alpha(n-1)I_{\phi(n)+1}+(1-\alpha)(J-I)_{\phi(n)+1}$ and $\mathscr{H}_i=\alpha\left(\phi(n)+\phi(d_i)+\displaystyle\sum_{k=1,k\neq i}^{t}\frac{\phi(d_k)}{d(v_i,v_k)}\right) \linebreak I_{\phi(d_i)}+(1-\alpha)(J-I)_{\phi(d_i)}$, for $1 \leq i \leq t$, we get the required result. \hfill$\blacksquare$

\begin{rem}For \( n = p^m \), where \( m \) is a positive integer, the power graph \( P(\mathbb{Z}_n) \) is isomorphic to the complete graph \( K_n \). As a result, the generalized reciprocal distance spectrum, the generalized distance spectrum, and the \( A_\alpha \)-spectrum of \( P(\mathbb{Z}_n) \) all coincide and are given by $\{n-1,(n\alpha -1)^{n-1}\}$.

\end{rem}
\begin{cor}
If \( n = pq \), where \( p \) and \( q \) are distinct primes with \( p < q \), then the generalized reciprocal distance spectrum of the power graph \( P(\mathbb{Z}_n) \) consists of the eigenvalue \( n\alpha - 1 \) with multiplicity \( \phi(pq) \), the eigenvalue \( \left(pq - \frac{q - 1}{2}\right)\alpha - 1 \) with multiplicity \( p - 2 \), and the eigenvalue \( \left(pq - \frac{p - 1}{2}\right)\alpha - 1 \) with multiplicity \( q - 2 \). The remaining three eigenvalues are given by the eigenvalues of the quotient matrix \( \mathcal{Q}^{RD_\alpha} \), which is defined as  
\[
\mathcal{Q}^{RD_\alpha} =
\begin{pmatrix}
   \left(n - 1 - \phi(pq)\right)\alpha + \phi(pq) & (1 - \alpha)\phi(p) & (1 - \alpha)\phi(q) \\
   (1 - \alpha)\phi(p) & \left(pq - p - \frac{q - 3}{2}\right)\alpha + p - 2 & \dfrac{(1 - \alpha)(q - 2)}{2} \\
   (1 - \alpha)\phi(q) & \dfrac{(1 - \alpha)(p - 2)}{2} & \left(pq - q - \frac{p - 3}{2}\right)\alpha + q - 2
\end{pmatrix}.
\]
\end{cor}

\noindent{\textbf{Proof.}} For a given \( n = pq \), where \( p \) and \( q \) are distinct primes, the proper divisors of \( n \) are \( p \) and \( q \). The complete subgraphs \( K_{\phi(p)} \) and \( K_{\phi(q)} \) of the power graph \( P(\mathbb{Z}_n) \) correspond to the elements generated by these divisors. Since all generators of \( \mathbb{Z}_n \) along with the identity element are adjacent to every other element in the graph, the power graph of \( \mathbb{Z}_n \) can be expressed as
\[
P(\mathbb{Z}_n) = K_{\phi(pq)+1} \left(K_{\phi(p)} \cup K_{\phi(q)}\right) = K_{1,2}[K_{\phi(pq)+1}, K_{\phi(p)}, K_{\phi(q)}].
\]
Using an appropriate labeling of the vertices, the generalized reciprocal distance matrix of \( P(\mathbb{Z}_n) \) is given by
\[
RD_\alpha(P(\mathbb{Z}_n)) = \begin{pmatrix}
    \mathscr{C}_1 & (1-\alpha)(p-1) & (1-\alpha)(q-1) \\
    (1-\alpha)(\phi(pq)+1) & \mathscr{C}_2 & \dfrac{(1-\alpha)(q-1)}{2} \\
    (1-\alpha)(\phi(pq)+1) & \dfrac{(1-\alpha)(p-1)}{2} & \mathscr{C}_3
\end{pmatrix},
\]
where
$\mathscr{C}_1 = (n-1)\alpha I_{\phi(pq)+1} + (1-\alpha)(J - I)_{\phi(pq)+1}$, $\mathscr{C}_2 = \left(pq - \frac{q+1}{2}\right)\alpha I_{\phi(p)} + (1-\alpha)(J - I)_{\phi(p)}$, and $\mathscr{C}_3 = \left(pq - \frac{p+1}{2}\right)\alpha I_{\phi(q)} + (1-\alpha)(J - I)_{\phi(q)}.$

Now, proceeding as in Theorem~\ref{t2.3}, the result follows.\hfill$\blacksquare$

\subsection{\texorpdfstring{$RD_\alpha$}{RD-alpha}-spectrum of the power graph of \texorpdfstring{$D_{2n}$}{(D2n)}}
The dihedral group \( D_{2n} \), which has \( 2n \) elements, is defined by  
$D_{2n} = \langle r, s \mid r^n = s^2 = 1,\ rs = sr^{-1} \rangle.$
The element \( r \) generates a cyclic subgroup \( \langle r \rangle \) of order \( n \), which is isomorphic to \( \mathbb{Z}_n \). The remaining \( n \) elements of the form \( sr^i \) (for \( i = 0, 1, \dots, n-1 \)) only share a power relationship with the identity element \( e \). So, in the power graph \( P(D_{2n}) \), these elements form an independent set, each connected only to the identity. This means the power graph of \( D_{2n} \) can be built by taking the power graph of \( \mathbb{Z}_n \), removing the identity, and then attaching \( n \) new pendant vertices (corresponding to the \( sr^i \)'s) to the identity vertex. Hence, the structure of the power graph is given by

\begin{equation}\label{2}
P(D_{2n})=\mathscr{G}[K_1, K_{\phi(n)},K_{\phi(d_1)},K_{\phi(d_2)},\ldots,K_{\phi(d_t)},\overline{K_n}]
,\end{equation}where the graph $\mathscr{G}$ is given in the following Figure 1.

\begin{center}
\begin{tikzpicture}[thick, scale=.8]
\node (T) at (2,3) [circle, fill=black, inner sep=1.5pt, label=right:{$\mathscr{T}$}] {};
\node (K1) at (-2.5,3) [circle, fill=black, inner sep=1.5pt, label=left:{$K_1$}] {};
\node (Kphi) at (-0.4,4) [circle, fill=black, inner sep=1.5pt, label=above:{$K_{\phi(n)}$}] {};
\node (Knbar) at (-2.4,1.5) [circle, fill=black, inner sep=1.5pt, label=left:{$\overline{K}_n$}] {};

\draw (K1) -- (T);
\draw (Kphi) -- (T);
\draw (K1) -- (Kphi);
\draw (Kphi) -- (K1);
\draw (Knbar) -- (K1);

\node at (0, 0.4) {\small \textbf{Figure 1}: Graph $\mathscr{G}$};
\end{tikzpicture}
\end{center}
The following theorem describes the generalized reciprocal distance spectra of the power graph of the dihedral group \( D_{2n} \).

\begin{thm}\label{t3.2}
The generalized reciprocal distance spectra of the power graph of the dihedral group \( D_{2n} \) consist of the eigenvalues $\frac{3n}{2}\alpha - 1$, $(n+1)\alpha-1$ and $\left(\phi(n) + \phi(d_i) + \sum_{\substack{k=1 \\ k\neq i}}^{t} \frac{1}{d(v_i,v_k)} + \frac{n}{2} + 1\right)\alpha - 1$ with multiplicities $\phi(n)-1$, $n-1$, and $\phi(d_i)-1$, for $1 \leq i \leq t$, respectively. The remaining $t+3$ eigenvalues are given by the eigenvalues of the quotient matrix \( \mathcal{Q}^{RD_\alpha} \), which is defined as
\[
\begin{pmatrix}
     (2n-1)\alpha & (1-\alpha)\phi(n) & (1-\alpha)\phi(d_1) & \cdots & (1-\alpha)\phi(d_t) & (1-\alpha)n\\
     1-\alpha & h & (1-\alpha)\phi(d_1) & \cdots & (1-\alpha)\phi(d_t) & \frac{(1-\alpha)n}{2}\\
     1-\alpha & (1-\alpha)\phi(n) & h_1 & \cdots & \frac{(1-\alpha)\phi(d_t)}{d(v_1,v_t)} & \frac{(1-\alpha)n}{2}\\
     \vdots & \vdots & \vdots & \ddots & \vdots & \vdots \\
     1-\alpha & (1-\alpha)\phi(n) & \frac{(1-\alpha)\phi(d_1)}{d(v_t,v_1)} & \cdots & h_t & \frac{(1-\alpha)n}{2}\\
     1-\alpha & \frac{(1-\alpha)\phi(n)}{2} & \frac{(1-\alpha)\phi(d_1)}{2} & \cdots & \frac{(1-\alpha)\phi(d_t)}{2} & h_{t+1}
\end{pmatrix},
\]
with $h= \left(\frac{3n}{2} - \phi(n)\right)\alpha + \phi(n) - 1$,
$h_i = \left(\phi(n) + \sum_{\substack{k=1 \\ k \neq i}}^t \frac{1}{d(v_i, v_k)} + \frac{n}{2} + 1\right)\alpha + \phi(d_i) - 1$, and $h_{t+1} = \frac{(n+1)\alpha}{2} + \frac{n-1}{2}$.
\end{thm}

\noindent{\textbf{Proof.}}
By Equation~\eqref{2}, the power graph of the dihedral group \( D_{2n} \) can be expressed as  
\[
P(D_{2n}) = \mathscr{G}[K_1, K_{\phi(n)}, K_{\phi(d_1)}, K_{\phi(d_2)}, \ldots, K_{\phi(d_t)}, \overline{K_n}],
\]
where \( \mathscr{G} \) is the graph shown in Figure~1, with vertex set \( \{v, v', v_1, \ldots, v_t, v_{t+1} \} \). Here, \( t \) denotes the number of proper divisors of \( n \) associated with the subgroup \( H = \langle r \rangle \subset D_{2n} \). By choosing a suitable labeling of the vertices in the power graph of \( D_{2n} \), the generalized reciprocal distance matrix of \( P(D_{2n}) \) is given as

\[
\small RD_\alpha(P(D_{2n}))=
 \bordermatrix{
      & 1 & \phi(n) & \phi(d_1) & \phi(d_2) & \cdots & \phi(d_t) & n \cr
   1   & (2n-1)\alpha & (1-\alpha) & (1-\alpha) & (1-\alpha) &  \cdots & (1-\alpha) & (1-\alpha)J \cr
 \phi(n)  & (1-\alpha) & \mathscr{H}' & (1-\alpha)J & (1-\alpha)J & \cdots & (1-\alpha)J & \frac{(1-\alpha)J}{2} \cr
  \phi(d_1) & (1-\alpha) & (1-\alpha)J & \mathscr{H}_1 & \frac{(1-\alpha)J}{d(v_1,v_2)} & \cdots & \frac{(1-\alpha)J}{d(v_1,v_t)} & \frac{(1-\alpha)J}{2} \cr
\phi(d_2)   & (1-\alpha) & (1-\alpha)J & \frac{(1-\alpha)J}{d(v_2,v_1)} & \mathscr{H}_2  & \cdots & \frac{(1-\alpha)J}{d(v_2,v_t)} & \frac{(1-\alpha)J}{2} \cr
\vdots   & \vdots & \vdots & \vdots & \vdots & \ddots & \vdots & \vdots \cr
  \phi(d_t)  &(1-\alpha) & (1-\alpha)J & \frac{(1-\alpha)J}{d(v_t, v_1)} & \frac{(1-\alpha)J}{d(v_t, v_2)} &\cdots & \mathscr{H}_t & \frac{(1-\alpha)J}{2} \cr
 n   & (1-\alpha) & \frac{(1-\alpha)J}{2} & \frac{(1-\alpha)J}{2} & \frac{(1-\alpha)J}{2} & \cdots & \frac{(1-\alpha)J}{2} & \mathscr{H}_{t+1}}, \]
    
 where $\mathscr{H}'=\left(\frac{3n}{2}-1\right)\alpha I_{\phi(n)}+(1-\alpha)(J-I)_{\phi(n)}$, $\mathscr{H}_i=\left(\phi(n)+\phi(d_i)+\sum_{k=1,k \neq i}^{t} \frac{1}{d(v_i,v_k)}+\frac{n}{2}\right) \linebreak \alpha I_{\phi(d_i)}+(1-\alpha)(J-I)_{\phi(d_i)}$ for $1 \leq i \leq t$, and $\mathscr{H}_{t+1}=n\alpha I_n + \frac{(1-\alpha)}{2}(J-I)_n$. Now, using Theorem \ref{t2.1}, we get the required result. \hfill$\blacksquare$

\begin{cor}
Let \( n = p^m \), where \( p \) is a prime number and \( m \in \mathbb{N} \). Then the generalized reciprocal distance spectra of the power graph of the dihedral group \( D_{2n} \) consist of the eigenvalues $(n+\frac{1}{2})\alpha-\frac{1}{2}$ and $\frac{3n}{2}\alpha-1$ with multiplicities $n-1$ and $n-2$, respectively. The remaining three eigenvalues are given by the eigenvalues of the quotient matrix \( \mathcal{Q}^{RD_\alpha} \), which is defined as
\[\mathcal{Q}^{RD_\alpha}=\begin{pmatrix}
  (2n-1)\alpha & (1-\alpha)(n-1) & (1-\alpha)n \\
  1-\alpha & (\frac{n}{2}+1)\alpha+n-2 & \frac{1-\alpha}{2}n\\
  1-\alpha & \frac{1-\alpha}{2}(n-1) & \alpha+n-1
\end{pmatrix}.\]
\end{cor}

\noindent{\textbf{Proof.}}
For \( n = p^m \), where \( p \) is a prime number and \( m \in \mathbb{N} \), it is easy to observe from Equation~\eqref{2} that the power graph of the dihedral group \( D_{2n} \) can be expressed as
\[
\mathcal{P}(D_{2n}) = K_{1,2}[K_1, K_{n-1}, \overline{K}_n].
\]
Using an appropriate labeling of the vertices, the generalized reciprocal distance matrix of \( P(D_{2n}) \) is given by 
\[RD_\alpha(P(D_{2n}))=\begin{pmatrix}
    (2n-1)\alpha & (1-\alpha)J & (1-\alpha)J\\
    (1-\alpha)J & \mathscr{H}_1 & \frac{(1-\alpha)J}{2}\\
    (1-\alpha)J & \frac{(1-\alpha)J}{2} & \mathscr{H}
_2\end{pmatrix},\] where $\mathscr{H}_1=(\frac{3}{2}n-1)\alpha I_{n-1}+(1-\alpha)(1-\alpha)(J-I)_{n-1}$ and $\mathscr{H}_2= n\alpha I_n+\frac{(1-\alpha)}{2}(J-I)_n$. Proceeding as in Theorem \ref{t2.1}, we get the required result. \hfill$\blacksquare$

\subsection{\texorpdfstring{$RD_\alpha$}{RD-alpha}-spectrum of the power graph of \texorpdfstring{$Q_{4n}$}{(Q4n)}}
The generalized quaternion group \( Q_{4n} \) of order \( 4n \) is defined by the presentation
\[
Q_{4n} = \langle r, s \mid r^{2n} = 1,\quad s^{2} = r^{n},\quad s^{-1} r s = r^{-1} \rangle.
\]
The following theorem describes the power graph of $Q_{4n}$ in terms of joined union graph.

\begin{thm} \label{T3.3}
Let \( Q_{4n} \) be the generalized quaternion group of order \( 4n \). Then the power graph \( \mathcal{P}(Q_{4n}) \) can be expressed as a generalized join
\[
\mathcal{P}(Q_{4n}) = \mathscr{G}[K_1, K_1, K_{\phi(2n)}, K_{\phi(d_1)}, K_{\phi(d_2)}, \ldots, K_{\phi(d_t)}, K_2, K_2, \ldots, K_2],
\]
where the parent graph \( \mathscr{G} \) has vertex set $V(\mathscr{G}) = \{v_1, v_2, v_3\} \cup \{w_i\}_{i=1}^t \cup \{y_k\}_{k=1}^n$ with adjacency relations $v_1 \sim u \ \text{for all } u \in V(\mathscr{G}) \setminus \{v_1\}, 
v_2 \sim v_3, 
v_2 \sim y_k  \ \text{for all } 1 \leq k \leq n, 
v_2 \sim w_i \  \text{if } 2 \mid d_i, 
w_i \sim w_j \ \text{if } d_i \mid d_j \text{ or } d_j \mid d_i, 
v_3 \sim w_i \ \text{for all } 1 \leq i \leq t,$ and $\{d_i\neq 2\}_{i=1}^t$ are the proper divisors of $2n$.
\end{thm}

\noindent{\textbf{Proof.}} The elements of \( Q_{4n} \) consist of the \( 2n \) elements of the cyclic subgroup \( \langle r \rangle \), and the remaining \( 2n \) elements are of the form \( r^k s \), where \( 0 \le k < 2n \). To analyze the structure of \( \mathcal{P}(Q_{4n}) \), we partition its vertex set into disjoint subsets according to element type and order. The identity element \( e \) forms a singleton clique, denoted \( K_1 \), and the element \( r^n \) (the only element of order 2) also forms a singleton \( K_1 \). The generators of the cyclic subgroup \( \langle r \rangle \) form a clique of size \( \phi(2n) \), denoted \( K_{\phi(2n)} \), and for each proper divisor \( d_i \ne 2 \) of \( 2n \), the set of elements of order \( d_i \) in \( \langle r \rangle \) forms a clique of size \( \phi(d_i) \), denoted \( K_{\phi(d_i)} \), where \( 1 \le i \le t \). Additionally, each pair \( \{r^k s, r^{k+n}s\} \) for \( 0 \le k < n \) forms a clique, denoted \( K_2 \), since each such pair consists of two elements that are powers of each other.
Thus, the power graph \( \mathcal{P}(Q_{4n}) \) can be realized as a Joined union
\[
\mathcal{P}(Q_{4n}) = \mathscr{G}[K_1, K_1, K_{\phi(2n)}, K_{\phi(d_1)}, \ldots, K_{\phi(d_t)}, K_2, \ldots, K_2],
\]
where the parent graph \( \mathscr{G} \) has \( t + n + 3 \) vertices, and its vertex set is given by $V(\mathscr{G}) = \{v_1, v_2, v_3\} \cup \{w_i\}_{i=1}^t \cup \{y_k\}_{k=1}^n$,
with adjacency relations precisely described in the theorem. This completes the proof.\hfill$\blacksquare$

The following theorem describes the generalized reciprocal distance spectra of the power graph of the generalized quaternion group \( Q_{4n} \).

\begin{thm}\label{T3.4}
The generalized reciprocal distance spectrum of \( \mathcal{P}(Q_{4n}) \) consists of the eigenvalues \( 3n\alpha - 1 \), \( (2n + 1)\alpha \), \( 2(n + 1)\alpha - 1 \), and 
$
\left( \phi(2n) + \phi(d_i) + n + \dfrac{1}{d(w_i, v_2)} + \sum_{j=1}^{t} \dfrac{1}{d(w_i, w_j)} + 1 \right)\alpha + 1,
$
with respective multiplicities \( \phi(2n) - 1 \), \( n \), \( n - 1 \), and \( \phi(d_i) - 1 \). The remaining \( t + 4 \) eigenvalues are determined by the quotient matrix \( \mathcal{Q}^{RD_\alpha} \) associated with the equitable partition of the power graph \( \mathcal{P}(Q_{4n}) \), given as 
$$\begin{pmatrix}
   (4n-1)\alpha & 1-\alpha & (1-\alpha)\phi(2n) & (1-\alpha)\phi(d_1) & (1-\alpha)\phi(d_2) & \cdots & (1-\alpha)\phi(d_t) & 2n(1-\alpha)\\
   (1-\alpha) & h & (1-\alpha)\phi(2n) & \frac{(1-\alpha)\phi(d_1)}{d(v_2,w_1)} & \frac{(1-\alpha)\phi(d_2)}{d(v_2,w_2)} & \cdots & \frac{(1-\alpha)\phi(d_t)}{d(v_2,w_t)} & 2n(1-\alpha)\\
   1-\alpha & 1-\alpha & h' & (1-\alpha)\phi(d_1) & (1-\alpha)\phi(d_2) & \cdots & (1-\alpha)\phi(d_t) & n(1-\alpha)\\
   1-\alpha & \frac{(1-\alpha)}{d(w_1,v_2)} & (1-\alpha)\phi(2n) & h_1 & \frac{(1-\alpha)\phi(d_2)}{d(w_1,w_2)} & \cdots & \frac{(1-\alpha)\phi(d_t)}{d(w_1,w_t)} & n(1-\alpha)\\
   1-\alpha & \frac{(1-\alpha)}{d(w_2,v_2)} & (1-\alpha)\phi(2n) & \frac{(1-\alpha)\phi(d_1)}{d(w_2,w_1)} & h_2 & \cdots & \frac{(1-\alpha)\phi(d_t)}{d(w_2,w_t)} & n(1-\alpha)\\
   \vdots & \vdots & \vdots & \vdots & \vdots & \ddots & \vdots & \vdots \\
   1-\alpha & \frac{(1-\alpha)\phi(d_t)}{d(w_t,v_2)} & (1-\alpha)\phi(2n) & \frac{(1-\alpha)\phi(d_1)}{d(w_t,w_1)} & \frac{(1-\alpha)\phi(d_2)}{d(w_t,w_2)} & \cdots & h_t & n(1-\alpha)\\
   1-\alpha & 1-\alpha & \frac{(1-\alpha)\phi(2n)}{2} & \frac{(1-\alpha)\phi(d_1)}{2} & \frac{(1-\alpha)\phi(d_2)}{2} & \cdots & \frac{(1-\alpha)\phi(d_t)}{2} & (n+1)\alpha +n
\end{pmatrix}
,$$ 
where $h=\left(2n+\phi(2n)+1+\sum_{i=1}^{t}\frac{1}{d(v_2,w_i)}\right)\alpha$, $h_i=\left( \phi(2n) + n + \dfrac{1}{d(w_i, v_2)} + \sum_{j=1}^{t} \dfrac{1}{d(w_i, w_j)} + 1 \right)\alpha +\phi(d_i)- 1$, for $1 \le i \le t$, $i \neq j$, and $h'=(3n-\phi(2n))\alpha+\phi(2n)-1$.
\end{thm}

\noindent{\textbf{Proof.}}
From Theorem \ref{T3.3}, the power graph of the generalized quaternion group \( Q_{4n} \) can be expressed as
\[
\mathcal{P}(Q_{4n}) = \mathscr{G}[K_1, K_1, K_{\phi(2n)}, K_{\phi(d_1)}, K_{\phi(d_2)}, \ldots, K_{\phi(d_t)}, K_2, K_2, \ldots, K_2],
\]
where \( \mathscr{G} \) is the parent graph defined in the preceding discussion. With an appropriate labeling of the vertices corresponding to this decomposition, the generalized reciprocal distance matrix \( RD_\alpha(\mathcal{P}(Q_{4n})) \) is given as

\[
\bordermatrix{
      & 1 & 1 & \phi(2n) & \phi(d_1) & \phi(d_2) & \cdots & \phi(d_t) & 2 & \cdots & 2 \cr
1   & (4n-1)\alpha & \beta & \beta J & \beta J & \beta J & \cdots & \beta J & \beta J & \cdots & \beta J \cr
1 & \beta & h & \beta J & \frac{\beta J}{d(v_2,w_1)} & \frac{\beta J}{d(v_2,w_2)} & \cdots & \frac{\beta J}{d(v_2,w_t)} & \beta J & \cdots & \beta J \cr
\phi(2n) & \beta J & \beta J & \mathscr{H}' & \beta J & \beta J & \cdots & \beta J & \frac{\beta J}{2} & \cdots & \frac{\beta J}{2} \cr
\phi(d_1) & \beta J & \frac{\beta J}{d(w_1,v_2)} & \beta J & \mathscr{H}_1 & \frac{\beta J}{d(w_1,w_2)} & \cdots & \frac{\beta J}{d(w_1,w_t)} & \frac{\beta J}{2} & \cdots & \frac{\beta J}{2} \cr
\phi(d_2) & \beta J & \frac{\beta J}{d(w_2,v_2)} & \beta J & \frac{\beta J}{d(w_2,w_1)} & \mathscr{H}_2 & \cdots & \frac{\beta J}{d(w_2,w_t)} & \frac{\beta J}{2} & \cdots & \frac{\beta J}{2} \cr
\vdots & \vdots & \vdots & \vdots & \vdots & \vdots & \ddots & \vdots & \vdots & \ddots & \vdots \cr
 \phi(d_t) & \beta J & \frac{\beta J}{d(w_t,v_2)} & \beta J & \frac{\beta J}{d(w_t,w_1)} & \frac{\beta J}{d(w_t,w_2)} & \cdots & \mathscr{H}_t & \frac{\beta J}{2} & \cdots & \frac{\beta J}{2} \cr
 2 & \beta J & \beta J & \frac{\beta J}{2} & \frac{\beta J}{2} & \frac{\beta J}{2} & \cdots & \frac{\beta J}{2} & \mathscr{A}_1 & \cdots & \frac{\beta J}{2}  \cr
 \vdots & \vdots & \vdots & \vdots & \vdots & \vdots & \ddots & \vdots & \vdots & \ddots & \vdots \cr
2 & \beta J & \beta J & \frac{\beta J}{2} & \frac{\beta J}{2} & \frac{\beta J}{2} & \cdots & \frac{\beta J}{2} & \frac{\beta J}{2} & \cdots & \mathscr{A}_n
}
\]

where $\beta=1-\alpha$, $h=\left(2n+\phi(2n)+1+\sum_{i=1}^{t}\frac{1}{d(v_2,w_i)}\right)\alpha$, $\mathscr{H}'=(3n-1)\alpha I _{\phi(2n)}+(1-\alpha)(J-I)_{\phi(2n)}$, and $\mathscr{H}_i=\left(\phi(2n)+\phi(d_i)+n+\frac{1}{d(w_i,v_2)}+\sum_{j=1}^{t}\frac{1}{d(w_i,w_j)}\right)\alpha I_{\phi(d_i)}+(1-\alpha)(J-I)_{\phi(d_i)}$, $1 \le i \le t$, $i \neq j$, and $\mathscr{A}_j=(2n+1)\alpha I_2+(1-\alpha)(J-I)_2$ for $1 \le j \le n$. 

Using Theorem \ref{t2.1}, we get the required result.
\hfill$\blacksquare$

\begin{cor}
Let $n=2^k$, $k \in \mathbb{N}$. Then the generalized reciprocal distance spectrum of the $\mathcal{P}(Q_{4n})$ consists of the eigenvalues \( 4n\alpha - 1 \), \( 3n\alpha - 1 \), \( 2(n+1)\alpha - 1 \), and \( (2n + 1)\alpha \), with respective multiplicities \( 1 \), \( 2n - 3 \), \( n \), and \( n - 1 \). The remaining three eigenvalues are obtained from the eigenvalues of the following quotient matrix associated with an equitable partition
\[
\begin{pmatrix}
(4n - 2)\alpha + 1 & 2(n - 1)(1 - \alpha) & 2n(1 - \alpha) \\
2(1 - \alpha) & (n + 2)\alpha + 2n - 3 & n(1 - \alpha) \\
2(1 - \alpha) & (n - 1)(1 - \alpha) & (n + 1)\alpha + n
\end{pmatrix}.
\]
\end{cor}

\noindent{\textbf{Proof.}}
By Theorem~\ref{T3.3}, the power graph of the generalized quaternion group \( Q_{4n} \), for $n=2^k$, $k \in \mathbb{N}$, is given by $$\mathcal{P}(Q_{4n})=K_{1,n+1}[K_2,K_{2n-2},K_2,K_2, \ldots, K_2].$$
The generalized reciprocal distance matrix is

\[RD_\alpha(P(Q_{4n}))=\bordermatrix{
      & 2 & 2n-2 & 2 & 2 & \cdots & 2 \cr
   2 & \mathscr{H}_1 & (1-\alpha)J & (1-\alpha)J & (1-\alpha)J & \cdots & (1-\alpha)J \cr
  2n-2 &  (1-\alpha)J & \mathscr{H}_2 & \frac{(1-\alpha)J}{2} & \frac{(1-\alpha)J}{2} & \cdots & \frac{(1-\alpha)J}{2} \cr
   2 & (1-\alpha)J & \frac{(1-\alpha)J}{2} & \mathscr{H}_3 & \frac{(1-\alpha)J}{2} & \cdots & \frac{(1-\alpha)J}{2} \cr
  2 & (1-\alpha)J & \frac{(1-\alpha)J}{2} & \frac{(1-\alpha)J}{2} & \mathscr{H}_4 & \cdots & \frac{(1-\alpha)J}{2} \cr
 \vdots &  \vdots & \vdots & \vdots &  \vdots & \ddots & \vdots  \cr
  2 & (1-\alpha)J & \frac{(1-\alpha)J}{2} & \frac{(1-\alpha)J}{2} & \frac{(1-\alpha)J}{2} & \cdots & \mathscr{H}_{n+2}
},\] where $\mathscr{H}_1=(4n-1)\alpha I_2+(1-\alpha)(J-I)_2$, $\mathscr{H}_2=(3n-1)\alpha I_{2n-2}+(1-\alpha)(J-I)_{2n-2}$, and $\mathscr{H}_i=(2n+1)\alpha I_2+(1-\alpha)(J-I)_2$, for $3 \le i \le n+2$.  Proceeding as in Theorem \ref{t2.1}, we get the required result. \hfill$\blacksquare$

\subsection{\texorpdfstring{$RD_\alpha$}{RD-alpha}-spectrum of the power graph of an elementary abelian \texorpdfstring{$p$}{p}-group}

A group \( \mathcal{G} \) is called an \emph{elementary abelian \( p \)-group} if every non-identity element of \( \mathcal{G} \) has order \( p \), where \( p \) is a prime number. The following proposition describes the structure of the power graph of such a group.

\begin{prop}[\textbf{Chelvan and Sattanathan \cite{cs(2013)}}]\label{l3.1}
Let \( \mathcal{G} \) be an elementary abelian group of order \( p^n \) for some prime \( p \) and positive integer \( n \). Then,
\[
\mathcal{P}(\mathcal{G}) \cong K_1 \vee \left(\bigcup_{i=1}^{l} K_{p-1}\right), \quad \text{where } l = \frac{p^n - 1}{p - 1}.
\]
\end{prop}

Following the above result, we determine the generalized reciprocal distance spectra of the power graph of elementary $p$-group of order $p^n$.

\begin{thm}\label{t3.3}
 Let $\mathcal{G}$ be an elementary abelian $p$-group of order $p^n$ for some prime $p$ and positive integer $n$. Then the generalized reciprocal distance spectra of $P(\mathcal{G})$ consists of the eigenvalues $\left(\frac{(p-1)(l+1)}{2}+1\right)\alpha-1$ and $\left(\frac{2pl+p-2l+1}{2}\right)\alpha + \frac{p - 1}{2}$ with multiplicities $l(p-2)$ and $l-1$. The remaining $2$ eigenvalues are determine by the quotient matrix, given by $$\mathcal{Q}=\begin{pmatrix}
    l(p-1)\alpha & (1-\alpha)(p-1)(l-1) \\
    1-\alpha & \frac{(pl+p-l+1)\alpha}{2}+\frac{(p-1)(l-1)}{2}
 \end{pmatrix}.$$
\end{thm}

\noindent{\textbf{Proof.}} Let $\mathcal{G}$ be an elementary abelian $p$-group of order $p^n$ for the given $p$ and $n$. Then, by Proposition \ref{l3.1}, we have $P(\mathcal{G}) \cong K_1 \vee (\mathop{\cup}\limits_{i=1}^{l} K_{p-1})$ and so it follows $$P(\mathcal{G}) \cong \Gamma_1 [K_1, \undermat{l}{K_{p-1},\ldots,K_{p-1}}],$$ where $\Gamma_1 = K_1 \vee \overline{K_l} = K_{1,l}$.

The generalized reciprocal distance matrix of $P(\mathcal{G})$ is given by

$$RD_\alpha(P(\mathscr{G}))=\bordermatrix{
      & 1 & p-1 & p-1 & \cdots & p-1 \cr
1 & l(p-1)\alpha & (1-\alpha)J & (1-\alpha)J & \cdots & (1-\alpha)J \cr
 p-1 & (1-\alpha)J & \mathscr{H}_1 & \frac{(1-\alpha)J}{2} & \cdots & \frac{(1-\alpha)J}{2} \cr
 p-1 & (1-\alpha)J & \frac{(1-\alpha)J}{2} & \mathscr{H}_2 & \cdots & \frac{(1-\alpha)J}{2} \cr
 \vdots & \vdots & \vdots & \vdots & \ddots &  \vdots \cr
 p-1 & (1-\alpha)J & \frac{(1-\alpha)J}{2} & \frac{(1-\alpha)J}{2} & \cdots & \mathscr{H}_l}
  ,$$  where $\mathscr{H}_i=\frac{(p-1)(l+1)}{2}\alpha I_{p-1}+(1-\alpha)(J-I)_{p-1}$ for $1 \leq i \leq l$. 

Proceeding as in Theorem \ref{t2.1}, we get 
\begin{equation}\label{e3}
    Char(RD_\alpha(P(\mathscr{G})),x)=\left(x-\left(\frac{(p-1)(l+1)}{2}+1\right)\alpha+1\right)^{l(p-2)} Char(\mathcal{Q}^{RD_\alpha(P(\mathcal{G}))},x),
\end{equation}
where $\mathcal{Q}^{RD_\alpha(P(\mathcal{G}))}$ is the corresponding quotient matrix given as

$$\mathcal{Q}^{RD_\alpha(P(\mathcal{G}))}=\left(
\begin{array}{c|cccc}
  l(p-1)\alpha & (1-\alpha)(p-1) & (1-\alpha)(p-1) & \cdots & (1-\alpha)(p-1)\\\hline
  1-\alpha & h_1 & \frac{(1-\alpha)(p-1)}{2} & \cdots & \frac{(1-\alpha)(p-1)}{2}\\
  1-\alpha & \frac{(1-\alpha)(p-1)}{2} & h_2 & \cdots & \frac{(1-\alpha)(p-1)}{2}\\
  \vdots & \vdots & \vdots & \ddots & \vdots \\
  1-\alpha & \frac{(1-\alpha)(p-1)}{2} & \frac{(1-\alpha)(p-1)}{2} & \cdots & h_l
\end{array}
\right),$$ with $h_i=((p-1)l+1)\alpha+p-1$ for $1 \leq i \leq l$.

To determine some eigenvalues of $\mathcal{Q}^{RD_\alpha(P(\mathcal{G}))}$, we define the vectors $\mathscr{X}_i = \mathbf{e}_{i+1} - \mathbf{e}_{i+2}$ for $1 \le i \le l - 1$, where each vector $\mathscr{X}_i \in \mathbb{R}^{l+1}$ and $e_i$ is the standard basis of $\mathbb{R}^{l+1}$. Then, we have
\begin{equation}\label{e4}
\mathcal{Q}^{RD_\alpha(P(\mathcal{G}))} \mathscr{X}_i = \left(\left(\frac{2pl+p-2l+1}{2}\right)\alpha + \frac{p - 1}{2} \right)\mathscr{X}_i, \quad 1 \le i \le l - 1.
\end{equation}
This implies that $\left(\frac{2pl+p-2l+1}{2}\right)\alpha + \frac{p - 1}{2}$ is an eigenvalue of $\mathcal{Q}$ with multiplicity $l - 1$. The remaining two eigenvalues of $\mathcal{Q}^{RD_\alpha(P(\mathcal{G}))}$ are determined from the following quotient matrix
\begin{equation}\label{e5}
\mathcal{Q}=\begin{pmatrix}
    l(p-1)\alpha & (1-\alpha)(p-1)(l-1) \\
    1-\alpha & \frac{(pl+p-l+1)\alpha}{2}+\frac{(p-1)(l-1)}{2}
 \end{pmatrix}.\end{equation}

By Equations \eqref{e3}, \eqref{e4}, and \eqref{e5} we obtain the generalized reciprocal distance spectrum of the power graph of an elementary abelian $p$-group of order $p^n$. \hfill$\blacksquare$

\subsection{\texorpdfstring{$RD_\alpha$}{RD-alpha}-spectrum of the power graph of a non-abelian group of order \texorpdfstring{$pq$}{pq}}

The following proposition describes the power graph of a non-abelian group of order $pq$.
	
\begin{prop}{\bf (Chelvan and Sattanathan \cite{cs(2013)})\label{l3.2}}
Let $\mathcal{G}$ be a finite group of order $pq$, where $p$ and $q$ are primes. Then $\mathcal{G}$ is a non-abelian group if and only if  $P(\mathcal{G}) \cong K_1 \vee (qK_{p-1} \cup K_{q-1}).$ 
\end{prop}

Using the above result, we derive the generalized reciprocal distance spectrum for a non-abelian group whose order is the product of two distinct primes.

\begin{thm}\label{t3.4}
 Let \( \mathcal{G} \) be a non-abelian group of order \( pq \), where \( p \) and \( q \) are distinct primes with \( p < q \). The generalized reciprocal distance spectrum of the power graph \( \mathcal{P}(\mathcal{G}) \) includes the eigenvalues $\left(\frac{pq + p + 1}{2}\right)\alpha - 1,  \quad \left(\frac{(p + 1)q}{2}\right)\alpha - 1 \quad \text{and} \quad \left(\frac{pq+2}{2}\right)\alpha+\frac{p-3}{2}$ with multiplicities \( q(p - 2) \), \( q - 2 \), and $q-1$, respectively. The remaining \( 3 \) eigenvalues are determined from the following quotient matrix
 \[\begin{pmatrix}
     (pq-1)\alpha & (1-\alpha)(p-1)(q-1) & (1-\alpha)(q-1)\\
     1-\alpha & \frac{(q+2)\alpha}{2}+\frac{(p-1)(q-1)}{2} & \frac{(1-\alpha)(q-1)}{2}\\
     1-\alpha & \frac{q(1-\alpha)(p-1)}{2} & \frac{(pq-q+2)\alpha}{2}+q-2
 \end{pmatrix}.\]
\end{thm}
\noindent{\textbf{Proof.}}
Let $\mathcal{G}$ be a non-abelian group of order $pq$, where $p$ and $q$ are primes with $p < q$. Using Proposition~\ref{l3.2}, the power graph of $\mathcal{G}$ can be written as
\[
\mathcal{P}(\mathcal{G}) = K_{1,q+1}\big[K_1,\ \underbrace{K_{p-1}, K_{p-1}, \ldots, K_{p-1}}_{q\text{ times}},\ K_{q-1}\big].
\]

The generalized reciprocal distance matrix of $P(\mathcal{G})$ is given by
\[RD_\alpha(\mathcal{P}(\mathcal{G}))=\bordermatrix{
      & 1 & p-1 & p-1 & \cdots & p-1 & q-1 \cr
    1 & (pq-1)\alpha & (1-\alpha)J & (1-\alpha)J & \cdots & (1-\alpha)J & (1-\alpha)J\cr
  p-1 & (1-\alpha)J & \mathscr{H}_1 & \frac{(1-\alpha)J}{2} & \cdots & \frac{(1-\alpha)J}{2} & \frac{(1-\alpha)J}{2} \cr
 p-1 & (1-\alpha)J & \frac{(1-\alpha)J}{2} & \mathscr{H}_2 & \cdots & \frac{(1-\alpha)J}{2} & \frac{(1-\alpha)J}{2} \cr
 \vdots &  \vdots & \vdots & \vdots & \ddots & \vdots & \vdots \cr
  p-1 & (1-\alpha)J & \frac{(1-\alpha)J}{2} & \frac{(1-\alpha)J}{2} & \cdots & \mathscr{H}_q & \frac{(1-\alpha)J}{2} \cr
 q-1 & (1-\alpha)J & \frac{(1-\alpha)J}{2} & \frac{(1-\alpha)J}{2} & \cdots & \frac{(1-\alpha)J}{2} & \mathscr{H}_{q+1}
},\] where $\mathscr{H}_i=\left(\frac{pq+p-1}{2}\right)\alpha I_{p-1}+(1-\alpha)(J-I)_{p-1}$, $1 \le i \le q $, and $\mathscr{H}_{q+1}=\left(\frac{pq+q}{2}-1\right)\alpha I_{q-1}+(1-\alpha)(J-I)_{q-1}$.

\smallskip
Proceeding as in Theorem~\ref{t2.1}, we obtain
\begin{equation}\label{3}
Char(RD_\alpha(\mathcal{P}(\mathcal{G})), x) =
\left(x - \frac{pq + p + 1}{2}\alpha + 1\right)^{q(p - 2)}
\left(x - \frac{(p + 1)q}{2}\alpha + 1\right)^{q - 2}
Char(\mathcal{Q}, x),
\end{equation}
where $\mathcal{Q}$ is the corresponding quotient matrix of order $q+2$, given as
\[
\mathcal{Q} =
\left(
\begin{array}{c|cccc|c}
 (pq-1)\alpha & (1-\alpha)(p-1) & (1-\alpha)(p-1) & \cdots & (1-\alpha)(p-1) & (1-\alpha)(q-1)\\ \hline
 1-\alpha & h_1 & \frac{(1-\alpha)(p-1)}{2} & \cdots & \frac{(1-\alpha)(p-1)}{2} & \frac{(1-\alpha)(q-1)}{2}\\
 1-\alpha & \frac{(1-\alpha)(p-1)}{2} & h_2 & \cdots & \frac{(1-\alpha)(p-1)}{2} & \frac{(1-\alpha)(q-1)}{2}\\
 \vdots & \vdots & \vdots & \ddots & \vdots & \vdots \\
 1-\alpha & \frac{(1-\alpha)(p-1)}{2} & \frac{(1-\alpha)(p-1)}{2} & \cdots & h_q & \frac{(1-\alpha)(q-1)}{2}\\ \hline
 1-\alpha & \frac{(1-\alpha)(p-1)}{2} & \frac{(1-\alpha)(p-1)}{2} & \cdots & \frac{(1-\alpha)(p-1)}{2} & h_{q+1}
\end{array}
\right),
\]

with
$h_i = \frac{pq - p + 3}{2}\alpha + p - 2 \quad \text{for } 1 \le i \le q, \quad \text{and} \quad h_{q+1} = \frac{pq - q + 2}{2}\alpha + q - 2.$

To determine some eigenvalues of $\mathcal{Q}$, we define vectors $\mathscr{X}_i = \mathbf{e}_{i+1} - \mathbf{e}_{i+2}$ for $1 \le i \le q - 1$, where each vector $\mathscr{X}_i \in \mathbb{R}^{q+2}$ and $e_i$ denotes the standard basis vector of $\mathbb{R}^{q+2}$. Then, we have
\begin{equation}\label{4}
\mathcal{Q} \mathscr{X}_i = \left(\left(\frac{pq + 2}{2}\right)\alpha + \frac{p - 3}{2} \right)\mathscr{X}_i, \quad 1 \le i \le q - 1.
\end{equation}
This implies that $\left(\frac{pq + 2}{2}\right)\alpha + \frac{p - 3}{2}$ is an eigenvalue of $\mathcal{Q}$ with multiplicity $q - 1$. The remaining three eigenvalues of $\mathcal{Q}$ are determined from the following quotient matrix
\begin{equation}\label{5}
\begin{pmatrix}
     (pq - 1)\alpha & (1 - \alpha)(p - 1)(q - 1) & (1 - \alpha)(q - 1) \\
     1 - \alpha & \frac{(q + 2)\alpha}{2} + \frac{(p - 1)(q - 1)}{2} & \frac{(1 - \alpha)(q - 1)}{2} \\
     1 - \alpha & \frac{q(1 - \alpha)(p - 1)}{2} & \frac{(pq - q + 2)\alpha}{2} + q - 2
\end{pmatrix}.
\end{equation}

By  Equations \eqref{3}, \eqref{4}, and \eqref{5} we obtain the generalized reciprocal distance spectrum of the power graph of a non-abelian group of order \( pq \). \hfill$\blacksquare$

\section{Conclusion}

In this paper, we studied the generalized reciprocal distance spectrum of the joined union graph and several of its special cases. We expressed the power graphs of the dihedral group and the generalized quaternion group in terms of the joined union graph. As an application of these results, we determined the generalized reciprocal distance spectra of power graphs associated with various finite groups, including the finite cyclic group $\mathbb{Z}_n$, the dihedral group $D_{2n}$, the generalized quaternion group $Q_{4n}$, elementary abelian $p$-groups, and non-abelian groups of order $pq$.

The results of this paper extend earlier studies and provide deeper insight into how the algebraic structure of a group influences the spectral properties of its power graph. In addition, the explicit formulas obtained here allow the spectra of large graphs formed through standard graph operations to be computed efficiently. 

\bigskip
\noindent{\bf Acknowledgments:} The first author is grateful to DST INSPIRE (Award No. 03/2022/002991), New Delhi, India, for financial support.
	
\bigskip
\noindent{\bf Conflict of interest:} The authors have no conflict of interest to publish.
	
\bigskip
\noindent{\bf Data Availability Statement:} The authors declare that no data were used or analyzed.


\begin{thebibliography}{99}


\bibitem{abr(2019)}  A. Alhevaz, M. Baghipur, H. S. Ramane, Computing the reciprocal distance signless Laplacian eigenvalues and energy of graphs,
{\it Matematiche}, 74, 49–73, 2019.

\bibitem{bh(2010)} A. E. Brouwer, W. H. Haemers, Spectra of graphs, New York (NY), {\it Springer}, 2010.

\bibitem{bp(2018)} R. Bapat, S. K. Panda, The spectral radius of the reciprocal distance Laplacian matrix of a graph, {\it Bull. Iran. Math. Soc.}, 44,
1211–1216, 2018.

\bibitem{bkps(2018)} S. Barik, D. Kalita, S. Pati, G. Sahoo, Spectra of graph resulting from various graph operations and products: a survey, {\it Spec. Matrices}, 6(1), 323-342, 2018.

\bibitem{ba(2023)}
S. Banerjee, A. Adhikari, On spectra of power graphs of finite cyclic and dihedral groups, {\it Rocky Mt. J. Math.}, 53(2), 341-356, 2023.  

\bibitem{b(2023)} S. Banerjee, Distance Laplacian spectra of various graph operations and its application to graphs on algebraic structures, {\it Journal of Algebra Appl.}, 22(1), 2023.

\bibitem{cdt(1980)} D. M. Cvetkovic, M. Doob, N. Trinajstic, and H. Sachs, Spectra of Graphs: Theory and Applications,  Academic Press, 12(4), New York, 1980.

\bibitem{cgs(2009)}
I. Charkrabarty, M. Ghosh,  M. K. Sen, Undirected power graph of semigroups, {\it Semigroup Forum}, 78, 410-426, 2009.

\bibitem{cg(2011)} P. J. Cameron, Shamik Ghosh, The power graph of a finite group, {\it Discrete Mathematics},  311(13), 1220-1222, 2011.

\bibitem{cs(2013)}
T. T. Chelvan, M. Sattanathan, Power graphs of finite abelian groups, {\it Algebra Discrete Math.}, 16(1), 33-41, 2013.

\bibitem{jmnt(2007)} D. Janezic, A. Milicevic, S. Nikolic, and N. Trinajstic, Graph-Theoretical Matrices in Chemistry, Uni Kragujevac, Kragujevac, 2007.

\bibitem{mga(2017)} Z. Mehranian, A. Gholami, A. R. Ashrafi, The spectra of power graphs of certain finite groups, {\it Linear Multilinear Algebra}, 65, 1003-1010, 2017.

\bibitem{mt(2021)} L. Medina, M. Trigo, Upper bounds and lower bounds for the spectral radius of reciprocal distance, reciprocal distance Laplacian, and reciprocal distance signless Laplacian matrices, {\it Linear Algebra Appl.}, 609, 386-412, 2021.

\bibitem{pntm(1993)} D. Plav\v{s}i{\'c}, S. Nikoli{\'c}, N. Trinajsti{\'c}, Z. Mihali{\'c}, On the Harary index for the characterization of chemical graphs, {\it J. Math. Chem.}, 12, 235-250, 1993.

\bibitem{bp(2023)} 
B. A. Rather, S. Pirzada, On distance Laplacian spectra of certain finite groups, {\it Acta Math. Sin. Engl. Ser.}, 39(4), 603-617, 2023.

\bibitem{bgp(2023)} B. A. Rather, H. A. Ganie, S. Pirzada, On $A_\alpha$-spectrum of joined union of graphs and its applications to power graphs of finite groups, {\it Journal of Algebra Appl.}, 22(12), 2023.

\bibitem{stpa(2025)} Y. Singh, A. K. Tiwari, M. S. Pandey, F. Ali, On the power graphs over gyrogroups, {\it Asian-Eur. J. Math.}, 2025. https://doi.org/10.1142/S1793557125400121.

\bibitem{ss(2025)} A. Singh, Y. Singh, A. K. Tiwari, S. Pirzada, On the $A_{\alpha}$ matrix over finite metacyclic groups, {\it Proc. Natl. Acad. Sci., India, Sect. A Phys. Sci.}, 2025.

\bibitem{tcc} G-X. Tian, M-J. Chen, S-Y. Cui, The generalized reciprocal distance matrix of graphs, arXiv:2204.03787v1.

\bibitem{zt(2008)} B. Zhou, N. Trinajsti{\'c}, Maximum eigenvalues of the reciprocal distance matrix and the reverse Wiener matrix,  {\it Int. J. Quant. Chem.}, 108, 858-864, 2008.
\end{thebibliography}
\end{document}